\documentclass[10pt, reqno]{amsart}

\usepackage{amsmath,amssymb,amsthm, epsfig}
\usepackage{mathalfa}
\usepackage{amsfonts}
\usepackage{amsmath}
\usepackage{amssymb}
\usepackage[mathscr]{euscript}

\def \T {\mathbb{N}_0}

\def \suchthat {\ \big | \ }

\newtheorem{theorem}{Theorem}
\newtheorem{lemma}[theorem]{Lemma}
\newtheorem{proposition}[theorem]{Proposition}
\newtheorem{corollary}[theorem]{Corollary}

\theoremstyle{definition}

\theoremstyle{remark}

\numberwithin{equation}{section}

\title[The $C^{p^\prime}$-regularity conjecture in the plane]{A proof of the $C^{p^\prime}$-regularity conjecture\\ in the plane}

\author[D.J. Ara\'ujo]{Dami\~ao J. Ara\'ujo}
\address{Universidade Federal da Para\'iba, Department of Mathematics, Jo\~ao Pessoa 58.051-900, Brazil}{}
\email{araujo@mat.ufpb.br}

\author[E.V. Teixeira]{Eduardo V. Teixeira}
\address{University of Central Florida, Department of Mathematics, Orlando, FL, 32828}{}
\email{eduardo.teixeira@ucf.edu}

\author[J.M.~Urbano]{Jos\'{e} Miguel Urbano}
\address{CMUC, Department of
Mathematics, University of Coimbra, 3001-501 Coimbra, Portugal}{}
\email{jmurb@mat.uc.pt}

\begin{document}



\begin{abstract} 
We establish a new oscillation estimate for solutions of nonlinear partial differential equations of elliptic, degenerate type. This new tool yields a precise control on the growth rate of solutions near their set of critical points, where ellipticity degenerates. As a consequence, we are able to prove the planar counterpart of the longstanding conjecture that solutions of the degenerate $p$-Poisson equation with a bounded source are locally of class $C^{p^\prime}=C^{1,\frac{1}{p-1}}$; this regularity is optimal.

\end{abstract}

\date{\today}

\maketitle

\section{Introduction} \label{sct intro}

In this paper we investigate sharp $C^{1,\alpha}$-regularity estimates for solutions of the degenerate elliptic equation, with a bounded source,
\begin{equation}\label{p-poisson}
	-\Delta_p u = f(x) \in L^\infty(B_1), \qquad p>2.
\end{equation}
Establishing optimal regularity estimates is quite often a delicate matter and, in particular, $f(x) \in L^\infty$ is known to be a borderline condition for regularity. In the linear, uniformly elliptic case $p=2$, solutions of 
$$-\Delta u = f(x) \in L^\infty(B_1)$$ 
are locally in $C^{1,\alpha}$, for every $\alpha \in (0,1)$, but may fail to be in $C^{1,1}$. Obtaining such an estimate in specific situations, like free boundary problems, often involves a deep and fine analysis. 

In the degenerate setting $p>2$, the smoothing effects of the operator are far less efficient. Nonetheless, it is well established, see for instance \cite{DiB, Tol}, that a weak solution to \eqref{p-poisson} is locally of class $C^{1,\beta}$, for some exponent $\beta>0$ depending on dimension and $p$. If $p^\prime$ denotes the conjugate of $p$, i.e.,
$$p + p^\prime= p p^\prime ,$$ 
the radial symmetric example
$$	-\Delta_p \left (  c_p |x|^{p^\prime} \right )= 1$$
sets the limits to the optimal regularity and gives rise to the following well known open problem among experts in the field.

\bigskip

\noindent {\bf Conjecture} ($C^{p^\prime}$-regularity conjecture). {\it Solutions to \eqref{p-poisson} are locally of class $C^{1,\frac{1}{p-1}} = C^{p^\prime}$.} 

\bigskip

This problem touches very subtle issues in regularity theory. As  mentioned above, the conjecture is not true in the linear, uniformly elliptic setting, $p=2$, where merely $C^{1, \text{LogLip}}$-estimates are possible.  Notice further that a positive answer implies that $|x|^{p^\prime}$ -- a function whose $p$-laplacian is constant (real analytic) -- is the least regular among all functions whose $p$-laplacian is bounded. This is, at first sight, counterintuitive. 

We show in this paper that the conjecture holds true provided $p$-harmonic functions, which are the solutions of the homogeneous counterpart of \eqref{p-poisson}, are locally uniformly of class $C^{1,\alpha}$, with 
$$\alpha> \displaystyle \frac{1}{p-1}.$$
While this is still open in higher dimensions, it holds true in the plane, thus yielding a full proof of the conjecture in 2-$d$. The crucial estimate follows from results by Baernstein II and Kovalev in \cite{BK}, exploiting the fact that the complex gradient of a  $p$-harmonic function in the plane is a $K$-quasiregular gradient mapping. In a somewhat related issue, let us mention that, yet in the plane, Evans and Savin proved in \cite{ES} (see also \cite{Savin}) that infinity harmonic functions, i.e., viscosity solutions of 
$$
	\Delta_\infty u := u_{x_i x_j} u_{x_i} u_{x_j}=0,
$$ 
are locally of class $C^{1,\gamma}$ for some $0< \gamma \ll 1$. Whether infinity harmonic functions are of class $C^1$ in higher dimensions is still a major open problem in the field. 

We next state the main result of this paper.

\begin{theorem}\label{main thm} Let  $B_1 \subset \mathbb{R}^2$, and let $u \in W^{1,p}(B_1)$ be a weak solution of 
$$-\Delta_p u = f(x), \qquad p>2,$$
with $f\in L^\infty(B_1)$. Then $u \in C^{p^\prime}(B_{1/2})$ and
$$
	\|u\|_{C^{p^\prime}(B_{1/2})} \le C_p \left( \|f\|^{\frac{1}{p-1}}_{L^\infty(B_1)} + \|u\|_{L^p(B_1)} \right).
$$
\end{theorem}

Similar results have been independently obtained by Lindgren and Lindqvist using somewhat different methods. In \cite{LL} (see also \cite{KM}), they show that planar solutions to \eqref{p-poisson} are locally of class $C^{p^\prime-\epsilon}$, for $\epsilon >0$. The lack of a global $C^{1,\alpha}$-estimate for $p$-harmonic functions in 2-$d$ was the main reason precluding the passing from this asymptotic version to the optimal result.

Our approach is based on a new oscillation estimate (see Theorem \ref{edu-crack}), which is interesting on its own and reveals some essential nuances of the problem, disclosing, for example, the conjecture in higher dimensions in a number of relevant scenarios (see \cite{ATU2}). It gives a precise control on the oscillation of a solution to \eqref{p-poisson} in terms of the magnitude of its gradient,
\begin{equation}\label{crack-intro}
	\sup\limits_{B_r} |u(x) - u(0)| \lesssim r^{p^\prime} + |\nabla u(0)| r,
\end{equation}
and yields, by geometric iteration, improved $C^{1,\alpha}$ regularity estimates.
The insight to obtain such a refined control comes from the striking results in \cite{T2}, where {\it improved} regularity estimates are obtained for degenerate equations precisely along their set of critical points, $\{\nabla u = 0\}$. 

We are convinced the set of ideas and insights included in this paper are robust and versatile, and will foster future developments of the theory, far beyond the $C^{p^\prime}$-regularity conjecture. For example, when implementing Caffarelli's geometric approach to the analysis of $C^{1,\alpha}$-estimates for degenerate problems, a key obstruction that frequently arises is that if $u$ is a solution and $\ell$ is an affine function, then no PDE is {\it a priori} satisfied by $(u-\ell)$. That is the case, for example, in \cite{IS}, where the  Ishii-Lions method was employed to control affine perturbations of the solution $u$. Our new oscillation estimate provides a definitive tool for treating this common issue in regularity theory.

The paper is organized as follows. To render the paper reasonably self-contained, we gather in section \ref{sct pre} the results concerning the regularity of $p$-harmonic functions in the plane that will be used in the proof of Theorem \ref{main thm}. In section \ref{sct corrector} we introduce $C^{1}$-small correctors that link the regularity theory for \eqref{p-poisson} to that of $p$-harmonic functions. The key, new oscillation estimate is delivered in section \ref{sct crack}, and in section \ref{sct soft} we conclude the proof of the main theorem.   

Throughout the paper, we say a constant is \textit{universal} if it only depends on $p$.

\section{Flatland regularity for $p$-harmonic functions} \label{sct pre}

In this section we revisit the $C^{1,\alpha}$ regularity theory for $p$-harmonic functions, i.e., solutions to the homogeneous equation
\begin{equation}\label{Hom Eq}
	-\Delta_p u = 0.
\end{equation} 
That $p$-harmonic functions are locally of class $C^{1,\alpha(d,p)}$, for some exponent $0 < \alpha(d,p)  < 1$, that depends on the dimension $d$ and the power exponent $p$, is known since the late 60's (see \cite{U}). Away from the set of critical points
$$\mathscr{S}(u):=\{x \suchthat \nabla u(x) = 0\},$$
$p$-harmonic functions are $C^\infty$-smooth; however $C_\text{loc}^{1,\alpha(d,p)}$  is in fact the best possible regularity class since, along $\mathscr{S}(u)$, the Hessian of a $p$-harmonic function may become unbounded. Very little, if anything, is known concerning the value of the optimal, sharp H\"older exponent $\alpha(d,p)$ when $d\ge 3$. 

In the plane, however, a remarkable result in \cite{IM}, due to Iwaniec and Manfredi, assures that any $p$-harmonic function is of class $C^{1,\alpha^\ast(2,p)}$, for
\begin{equation}\label{IM exp}
	\alpha^\ast (2,p) = \frac{1}{6} \left ( \frac{p}{p-1} + \sqrt{1 + \frac{14}{p-1} + \frac{1}{(p-1)^2}} \right ),
\end{equation}
and  this regularity is optimal. The proof exploits the fact that the complex gradient of a $p$-harmonic function is a $K$-quasiregular mapping and uses a hodograph transformation to linearize the problem. Unfortunately, no explicit estimates, yielding a local universal control of the $C^{1,\alpha^\ast}$-norm of $u$, are written down in \cite{IM}. Near singular points we would have to examine the hodograph inversion process and track down the precise dependence of the successive norms involved -- a quite delicate issue \cite{M}. 

But, gloriously, there is more. In the apparently unrelated paper \cite{BK}, concerning non-divergence elliptic equations in the plane, Baernstein II and Kovalev show that $K$-quasiregular \textit{gradient} mappings are of class $C_\text{loc}^{1,\alpha}$, for an exponent $\alpha$ depending only on $K$, and obtain uniform estimates. For the reader's convenience we revisit their proof applied to our case. 

Let $u \in W_{loc}^{1,p} (B_1)$ be $p$-harmonic, $p>2$. Its complex gradient $\phi=\partial u / \partial z$ turns out to be a $(p-1)$-quasiregular mapping. This means that $\phi \in W_{loc}^{1,2} (B_1)$, which follows from estimates in \cite{BI} (see also \cite{AIM}), and that
\begin{equation}
\left| \frac{\partial \phi(z)}{\partial \overline{z}} \right| \leq \left( 1-\frac{2}{p} \right) \left| \frac{\partial \phi(z)}{\partial z} \right| , \quad \mbox{a.e.  in } B_1, 
\label{K-qr}
\end{equation}
which follows from giving \eqref{Hom Eq} the form of the complex equation
$$\frac{\partial \phi}{\partial \overline{z}} = \left( \frac{1}{p} - \frac{1}{2} \right) \left( \frac{\overline{\phi}}{\phi} \frac{\partial \phi}{\partial z} + \frac{\phi}{\overline{\phi}} \overline{\frac{\partial \phi}{\partial z}} \right).$$
Here $z=x+iy$ is the complex variable and the operators of complex differentiation are defined by
$$\frac{\partial }{\partial z} = \frac{1}{2} \left( \frac{\partial }{\partial x} - i \frac{\partial }{\partial y}\right) \qquad {\rm and} \qquad  \frac{\partial }{\partial \overline{z}} = \frac{1}{2} \left( \frac{\partial }{\partial x} + i \frac{\partial }{\partial y}\right).$$
The complex gradient $\phi$ is also a gradient mapping, i.e.,
$${\rm Im} \: \frac{\partial \phi}{\partial \overline{z}} = 0,$$
which holds since
$$\frac{\partial \phi}{\partial \overline{z}} = \frac{\partial^2 u }{\partial \overline{z} \partial z} = \frac{1}{4} \Delta u.$$
This fact can be used to significantly improve the lower bound on the Jacobian of $\phi$ (cf. \cite[Lemma 2.1]{BK}) in 
$$J_\phi= \det \nabla \phi \geq \frac{1}{p-1} \left| \nabla \phi\right|^2,$$
which follows from \eqref{K-qr} alone. This is the crucial new ingredient to prove (see \cite[Section 2, and in particular inequality (2.5)]{BK}) that 
$$\int_{B_r} |\nabla \phi|^2 \leq (p-1)\: (2r)^{2\alpha(p)} \int_{B_{1/2}} |\nabla \phi|^2, \qquad 0<r\leq \frac{1}{2},$$
with
$$\alpha (p) = \frac{1}{2p} \left ( -3 -\frac{1}{p-1} +  \sqrt{33 + \frac{30}{p-1} + \frac{1}{(p-1)^2}}  \right ),$$
which gives, by Morrey's lemma \cite[Lemma 12.2]{GT}, that $\phi \in C_{\rm loc}^{0,\alpha (p) }$ and 
$$
	[\phi]_{C^{0,\alpha (p)}(B_{1/2})} \leq 2^{1+\alpha (p)}\sqrt{\frac{p-1}{\alpha(p)}} \|\nabla \phi\|_{L^2(B_{1/2})}.
$$
Now, using standard estimates (see \cite{LL} for example), we obtain
$$[\phi]_{C^{0,\alpha (p)}(B_{1/2})} \le C_p \| u\|_{L^\infty (B_1)},$$
for a universal constant $C_p$ that only depends on $p$.
Finally observe that, for any $p>2$, we indeed have 
\begin{equation} \label{alfi}
\alpha^\ast (2,p) >\alpha (p) > \displaystyle \frac{1}{p-1}.
\end{equation} 
We summarize these results in the following proposition, which will be crucial in the proof of our main result.

\begin{proposition} \label{IM} For any $p>2$, there exists $0 < \tau_0 < \frac{p-2}{p-1}$ such that $p$-harmonic functions in $B_1 \subset \mathbb{R}^2$ are locally of class $C^{p^\prime+\tau_0}$. Furthermore, if $u \in W^{1,p} (B_1) \cap C(B_1)$ is $p$-harmonic in the unit disk $B_1\subset \mathbb{R}^2$ then there is a constant $C_p$, depending only on $p$, such that
\begin{equation}
	[\nabla u]_{C^{0,\frac{1}{p-1}+\tau_0}( B_{1/2})} \le C_p \| u\|_{L^\infty (B_1)} .
\label{xuxu}
\end{equation}
\end{proposition}

In the next three sections we will provide a proof of the the $C^{p^\prime}$-regularity conjecture in the plane, i.e. a proof of Theorem \ref{main thm}.  

\section{Existence of $C^1$-small correctors} \label{sct corrector}

In this section, we show that if $u$ is a normalized solution of 
$$
	-\Delta_p u = f(x),
$$ 
and $\|f\|_\infty \ll 1$, then we can find a  $C^1$ corrector $\xi$, with $\|\xi\|_{C^1} \ll 1$, such that $u + \xi$ is $p$-harmonic. This will allow us to frame the $C^{p^\prime}$ conjecture into the formalism of the so called geometric tangential analysis, e.g. \cite{C1}, \cite{ART, AT} and \cite{T0, T1, T2, T3, T4, TU}. Here is the precise  statement.

\begin{lemma}\label{l. correc}
Let $u \in W^{1,p} (B_1)$ be a weak solution of $-\Delta_p u = f$ in $B_1$, with $\|u\|_\infty \leq 1$. Given $\epsilon >0$, there exists $\delta=\delta (p,d,\epsilon)>0$ such that if $\|f\|_\infty \leq \delta$ then we can find a corrector $\xi \in C^1(B_{1/2})$, with 
\begin{equation}
| \xi (x) | \leq \epsilon \quad \mathrm{and} \quad | \nabla \xi (x) | \leq \epsilon, \quad \mathrm{in} \ B_{1/2}
\label{sc1}
\end{equation}
such that
\begin{equation}
-\Delta_p (u+\xi) = 0\quad \mathrm{in} \ B_{1/2}.
\label{sc2}
\end{equation}
\end{lemma}

\begin{proof}
Suppose the result does not hold. We can then find $\epsilon_0 >0$ and sequences of functions $(u_j)$ and $(f_j)$ in $W^{1,p} (B_1)$ and $L^\infty (B_1)$, respectively, such that 
$$-\Delta_p u_j = f_j \quad \mathrm{in} \ B_1; \qquad \|u_j\|_\infty \leq 1; \qquad \|f_j\|_\infty \leq 1/j$$
but, nonetheless, for every $\xi \in C^1(B_{1/2})$ such that
$$-\Delta_p (u_j+\xi) = 0\quad \mathrm{in} \ B_{1/2},$$
we have either $| \xi (x_0) | > \epsilon_0$ or $| \nabla \xi (x_0) | > \epsilon_0$, for a certain $x_0 \in B_{1/2}$. 

From classical estimates for the $p$-Poisson equation, we can extract a subsequence, such that, upon relabelling, 
$$u_j \longrightarrow u_\infty$$
in $C^1 (B_{1/2})$ as $j \to \infty$. Passing to the limit in the pde, we obtain 
$$-\Delta_p u_\infty =0 \quad \mathrm{in} \ B_{1/2}, \qquad \mathrm{with} \quad \|u_\infty\|_\infty \leq 1.$$
Now, let $\xi_j := u_\infty - u_j$. For $j_\ast \gg 1$, we have
$$-\Delta_p (u_{j_\ast}+\xi_{j_\ast}) = -\Delta_p u_\infty  = 0\quad \mathrm{in} \ B_{1/2}$$
and 
$$
	| \xi_{j_\ast} (x) | \leq \epsilon_0 \quad \mathrm{and} \quad | \nabla \xi_{j_\ast} (x) | \leq \epsilon_0, \quad \forall x \in B_{1/2},
$$
thus reaching a contradiction.

\end{proof}

We conclude this section by commenting that in order to prove Theorem \ref{main thm} it is enough to establish it for normalized solutions with small RHS, i.e., with $\|f\|_\infty \le \delta_0$. Indeed, if $u$ verifies
$-\Delta_p u = f(x)$, with $f\in L^\infty$, then the function
$$
	v(x) := \dfrac{u(\theta   x)}{\|u\|_\infty}
$$
is obviously normalized and
$$
	-\Delta_p v = \dfrac{\theta^{p}}{\|u\|^{p-1}_\infty} f(\theta x).
$$
Thus, choosing
$$
	\theta := \sqrt[p]{\frac{\delta_0 \|u\|_{\infty}^{p-1}}{\|f\|_\infty}},
$$
$v$ satisfies \eqref{p-poisson}, with small RHS. Once Theorem \ref{main thm} is proven for $v$, it immediately gives the corresponding estimate for $u$. 

\section{Analysis on the critical set} \label{sct crack}

In this section, based on an iterative reasoning, we establish the main tool that allows us to prove the $C^{p^\prime}$ conjecture in the plane. The following result is the first step in the iteration.

\begin{lemma}\label{l. discrete} There exists $0<\lambda_0<1/2$ and $\delta_0 >0$ such that if $\|f\|_\infty \leq \delta_0$ and $u \in W^{1,p} (B_1)$ is a weak solution of $-\Delta_p u = f$ in $B_1$, with $\|u\|_\infty \leq 1$, then 
$$
	\sup_{x\in B_{\lambda_0}} \Big| u(x) - \left[ u(0) + \nabla u(0) \cdot x \right] \Big| \leq {\lambda^{p^\prime}_0}.
$$
\end{lemma}

\begin{proof}
Take $\epsilon >0$ to be fixed later, apply the previous lemma to find $\delta_0$ and, under the smallness assumption on $f$, a respective corrector $\xi$ satisfying \eqref{sc1} and \eqref{sc2}. As  $(u+\xi)$ is $p$-harmonic in $B_{1/2}$ and, in view of Proposition \ref{IM}, $(u+\xi) \in C^{p^\prime + \tau_0}$, we can estimate in $B_{\lambda_0} \subset B_{1/2}$,
\begin{eqnarray*}
	\left | u(x) - \left[ u(0) + \nabla u(0) \cdot x \right] \right | & \leq &  \left | (u+\xi)(x) - \left[ (u+\xi)(0) + \nabla (u+\xi)(0) \cdot x \right] \right | \\
	& & + |\xi (x) | +| \xi (0) | + | \nabla \xi (0) \cdot x  | \\
	& \leq &  C_p {\lambda^{p^\prime + \tau_0}_0} + 3 \epsilon.
\end{eqnarray*}
We are also using the smallness of the corrector, assured by Lemma \ref{l. correc}. In order to complete the proof, we now make universal choices. Initially we choose $\lambda_0 \ll 1/2$ such that 
$$
	C_p {\lambda^{p^\prime + \tau_0} _0}< \frac{1}{2}  {\lambda^{p^\prime}_0}.
$$
In the sequel, we take 
$$
	\epsilon = \frac{1}{6}  {\lambda^{p^\prime}_0},
$$
which determines the smallness assumption on $\|f\|_\infty$ -- constant $\delta_0>0$ in the statement of the current lemma -- through the conclusion of Lemma \ref{l. correc}. Lemma \ref{l. discrete} is proven.

\end{proof}

The conclusion of Lemma \ref{l. discrete} does not, {\it per se}, allow an iteration since no obvious pde is satisfied by $u+\ell$, when $\ell$ is an affine function. Nonetheless, it provides the following information on the oscillation of $u$ in $B_{\lambda_0}$.
 
\begin{corollary}\label{corollary1}
Under the assumptions of the previous lemma, 
$$
	\sup_{x\in B_{\lambda_0}} | u(x) - u(0) | \leq {\lambda^{p^\prime}_0} + | \nabla u(0) | \lambda_0.
$$ 
\end{corollary}

\begin{proof}
This is a immediate application of the triangle inequality.

\end{proof}

The idea is now to iterate Corollary \ref{corollary1} in dyadic balls, keeping a precise track on the magnitude of the influence of $|\nabla u(0)|$.

\begin{theorem} \label{edu-crack}
Under the same assumptions of Lemma \ref{l. discrete}, there exists a constant $C>1$, depending only on $p$, such that 
$$
	\sup_{x\in B_r} | u(x) - u(0) | \leq C r^{p^\prime}  \left( 1 + | \nabla u(0) | \, r^{\frac{1}{1-p}} \right),  
$$ 
holds for all $0<r\ll 1$.
\end{theorem}

\begin{proof} We proceed by geometric iteration. Consider the universal constants $\lambda_0$ and $\delta_0$ obtained in the previous Lemma \ref{l. discrete} and let
$$v(x) = \frac{u(\lambda_0 x) - u(0)}{{\lambda^{p^\prime}_0} + | \nabla u(0) | \lambda_0}, \qquad x \in B_1.$$
We have $\|v \|_\infty \leq 1$, $v(0)=0$, and
$$
	\nabla v (0) = \frac{\lambda_0}{{\lambda^{p^\prime} _0}+ | \nabla u(0) | \lambda_0} \, \nabla u (0).
$$
Also, we have 
$$-\Delta_p v = \frac{\lambda_0^p}{\left( {\lambda^{p^\prime}_0} + | \nabla u(0) | \lambda_0 \right)^{p-1}} \, f(\lambda_0 x) \leq \frac{\lambda_0^p}{ {\lambda^{p^\prime (p-1)}_0} } \, | f(\lambda_0 x) | \leq \delta_0,$$
which entitles $v$ to Corollary \ref{corollary1}. Thus
$$
	\sup_{x\in B_{\lambda_0}} | v(x) - v(0) | \leq {\lambda^{p^\prime}_0} + | \nabla v(0) | \lambda_0, 
$$ 
which reads
$$
	\sup_{x\in B_{\lambda_0}} \left| \frac{u(\lambda_0 x) - u(0)}{{\lambda^{p^\prime} _0}+ | \nabla u(0) | \lambda_0} \right| \leq {\lambda^{p^\prime}_0} + \left| \frac{\lambda_0}{{\lambda^{p^\prime}_0} + | \nabla u(0) | \lambda_0} \, \nabla u (0) \right| \lambda_0,
$$
and hence
\begin{eqnarray*}
	\sup_{x\in B_{\lambda_0^2}} \left| u(x) - u(0) \right| & \leq & {\lambda^{p^\prime} _0}\Big[  {\lambda^{p^\prime}_0} + | \nabla u(0) | \lambda_0 \Big]+ \left| \nabla u (0) \right| \lambda_0^2\\
& \leq & \left( {\lambda^2_0} \right)^{p^\prime} + \left| \nabla u (0) \right| \lambda_0^2  \left(  1 + {\lambda^{p^\prime-1}_0}  \right).	
\end{eqnarray*}
Iterating this process and defining
$$
	a_k := \sup_{x\in B_{\lambda_0^k}} |u(x)-u(0)|,
$$
we obtain the inequality
\begin{eqnarray*}
a_{k} & \leq & \left( {\lambda^k_0} \right)^{p^\prime} +  \left| \nabla u (0) \right| {\lambda^k_0} \: \sum_{i=0}^{k-1} \left(  {\lambda^{p^\prime-1}_0} \right)^i\\
& \leq & \left( {\lambda^k_0} \right)^{p^\prime} +  \left| \nabla u (0) \right| {\lambda^k_0} \: \sum_{i=0}^{\infty} \left(  {\lambda^{p^\prime-1}_0} \right)^i\\
& \leq & \left( {\lambda^k_0} \right)^{p^\prime} +  \left| \nabla u (0) \right|  \: \frac{{\lambda^k_0}}{1-{\lambda^{p^\prime-1}_0}} ,
\end{eqnarray*}
valid for every $k \in \T$.

Now, given $0<r\ll 1$, let $k \in \T$ be such that $\lambda_0^{k+1} < r \leq \lambda_0^k$. Then
\begin{eqnarray*}
\sup_{x\in B_r} \frac{| u(x) - u(0) |}{r^{p^\prime}} & \leq & \sup_{x\in B_{\lambda_0^k}} \frac{| u(x) - u(0) |}{(\lambda^{k+1}_0)^{p^\prime}} = \frac{a_k}{(\lambda^{k+1}_0)^{p^\prime}}\\
& \leq & \frac{\left( {\lambda^k_0} \right)^{p^\prime} +  \left| \nabla u (0) \right|  \: \frac{{\lambda^k_0}}{1-{\lambda^{p^\prime-1}_0}} }{(\lambda^{k+1}_0)^{p^\prime}}\\
& \leq & \lambda^{-p^\prime}_0  + C(\lambda_0, p^\prime) \left| \nabla u (0) \right|  \left( {\lambda^k_0} \right)^{1-p^\prime}  \\
& \leq & \lambda^{-p^\prime}_0  + C(\lambda_0, p^\prime) \left| \nabla u (0) \right|  r^{1-p^\prime}  \\
& \leq & C   \left( 1 + | \nabla u(0) | \, r^{\frac{1}{1-p}} \right),
\end{eqnarray*}
as desired. Observe that $\lambda_0$ is a universal constant and so is $C$.

\end{proof}

In accordance to \cite[Theorem 3]{T2}, Theorem \ref{edu-crack} provides the aimed regularity along the set of critical points of $u$, $|\nabla u|^{-1}(0)$.
In fact, when $| \nabla u(0) | \leq r^{\frac{1}{p-1}}$, Theorem \ref{edu-crack} gives
\begin{eqnarray*}
	\sup_{x\in B_r} \Big| u(x) - \left[ u(0) + \nabla u(0) \cdot x \right] \Big| &  \leq & \sup_{x\in B_r} | u(x) - u(0) | + |  \nabla u(0)| \, r \\
	& \leq & (C+1) r^{p^\prime}.
\end{eqnarray*}

In the next section we show how Theorem \ref{edu-crack} can be used in its full strength to yield $C^{p^\prime}$-regularity at any point, regardless of the value of $|\nabla u|$; it will be a softer analysis.

\section{Analysis on the set of non-degenerate points} \label{sct soft} 

We now analyze the oscillation decay around points where the gradient is large. Recall our ultimate goal is to show that
$$
	\sup_{x\in B_r} \Big| u(x) - \left[ u(0) + \nabla u(0) \cdot x \right] \Big| \leq C\, r^{p^\prime}, \quad \forall \, 0<r\ll 1.
$$
For large values of  $|\nabla u|$, the operator is uniformly elliptic and hence stronger estimates are available. Assume then $| \nabla u(0) | > r^{\frac{1}{p-1}}$, define $\mu:= | \nabla u(0) |^{p-1}$ and take
$$
	w(x):=\frac{u(\mu x)  - u(0)}{\mu^{p^\prime}}.
$$
Clearly 
$$
	w(0)=0, \quad |\nabla w(0)| =1 \quad \text{and } -\Delta_p w = f(\mu x) \in L^\infty.
$$ 
Moreover, from Theorem \ref{edu-crack}, it follows that 
$$
	\sup_{x\in B_1}  |w(x)| = \sup_{x\in B_\mu} \frac{|u(x)-u(0)|}{\mu^{p^\prime}} \leq C,
$$
since $\mu^{\frac{1}{p-1}} = | \nabla u(0) |$. From classical $C^{1,\alpha}$ regularity estimates,  there exists a radius $\rho_0$, depending only on the data, such that
$$
	|\nabla w (x)| \geq \frac{1}{2}, \quad \forall\, x \in B_{\rho_0}.
$$
This implies that, in $B_{\rho_0}$, $w$ solves a uniformly elliptic equation. In particular, we have  
$$
	w \in C^{1,\beta} (B_{\rho_0}), \quad \mbox{for some }  \ \frac{1}{p-1} \leq \beta <1.
$$
As an immediate consequence,
$$
	\sup_{x\in B_r} \Big| w(x) - \nabla w(0) \cdot x \Big| \leq C\, r^{1+\beta}, \quad \forall \, 0<r<\frac{\rho_0}{2}
$$
which, in terms of $u$, reads
$$
	\sup_{x\in B_r} \Big| \frac{u(\mu x)  - u(0)}{\mu^{p^\prime}} - \mu^{1-p^\prime}\nabla u(0) \cdot x \Big| \leq C\, r^{1+\beta}.
$$
Since $p^\prime \leq 1+\beta$,  we conclude 
$$
	\sup_{x\in B_{r}} \Big| u(x)  - [u(0)  + \nabla u(0) \cdot  x ] \Big| \leq C\, r^{p^\prime}, \quad \forall \, 0<r<\mu\frac{\rho_0}{2}.
$$
Finally, for $\mu\frac{\rho_0}{2} \leq r < \mu$, we have
\begin{eqnarray*}
\sup_{x\in B_r} \Big| u(x) - \left[ u(0) + \nabla u(0) \cdot x \right] \Big| &  \leq &\sup_{x\in B_{\mu}} \Big| u(x)  - [u(0)  + \nabla u(0) \cdot  x ] \Big|\\
& \leq & \sup_{x\in B_\mu} | u(x) - u(0) | + |  \nabla u(0)| \, \mu\\
& \leq & (C+1) \mu^{p^\prime}\\
& \leq & C \left( \frac{2r}{\rho_0} \right)^{p^\prime}\\
& = & C r^{p^\prime}.
\end{eqnarray*}

In view of the reduction discussed at the end of Section \ref{sct corrector}, the proof of Theorem \ref{main thm} is complete.  \hfill $\square$

\bigskip

\noindent{\bf Acknowledgments.} The authors would like to thank Erik Lindgren and Juan Manfredi for their valuable comments and suggestions and Wenhui Shi for bringing \cite{BK} to our attention. 

This work was developed in the framework of the Brazilian Program \textsl{Ci\^encia sem Fronteiras}. The second and third authors thank the hospitality of ICMC--Instituto de Ci\^encias Matem\'aticas e de Computa\c{c}\~ao, from Universidade de S\~ao Paulo in S\~ao Carlos, where this work was initiated. D.A. supported by CNPq. E.V.T. partially supported by CNPq and Fapesp. J.M.U. partially supported by the Centre for Mathematics of the University of Coimbra -- UID/MAT/00324/2013, funded by the Portuguese Government through FCT/MCTES and co-funded by the European Regional Development Fund through the Partnership Agree\-ment PT2020.

\bigskip

\bibliographystyle{amsplain, amsalpha}

\end{document}